\documentclass[11pt]{amsart}
\usepackage{verbatim}

\newtheorem{theorem}{Theorem}[section]
\newtheorem{proposition}[theorem]{Proposition}

\newtheorem{claim}[theorem]{Claim}

\newtheorem{cor}[theorem]{Corollary}
\newtheorem{question}[theorem]{Question} 
\newtheorem{definition}[theorem]{Definition}

\newtheorem{conjecture}[theorem]{Conjecture}

\theoremstyle{plain}
\numberwithin{equation}{theorem}

\theoremstyle{remark}
\newtheorem*{notation}{Notation}

\newcommand{\C}{{\mathbb C}}

\DeclareMathOperator{\nonsing}{nonsing}
\DeclareMathOperator{\Id}{Id}
\newcommand{\Cp}{\C_p}

\newcommand{\N}{{\mathbb N}}

\newcommand{\cA}{{\mathcal A}}

\newcommand{\D}{{\mathbb D}}
\newcommand{\PCp}{\bP^1(\bC_p)}

\DeclareMathOperator{\PGL}{PGL}
\newcommand{\Dbar}{\overline{\D}}

\newcommand{\add}{\oplus}
\newcommand{\minus}{\ominus}

\newcommand{\fO}{\mathfrak o}

\newcommand{\Kbar}{\overline{K}}

\DeclareMathOperator{\Spec}{Spec}

\DeclareMathOperator{\GL}{GL}

\DeclareMathOperator{\End}{End}

\DeclareMathOperator{\bN}{\mathbb{N}}

\newcommand{\bP}{{\mathbb P}}
\newcommand{\bZ}{{\mathbb Z}}
\newcommand{\bG}{{\mathbb G}}

\newcommand{\bC}{{\mathbb C}}

\newcommand{\bA}{{\mathbb A}}
\newcommand{\bQ}{{\mathbb Q}}
\newcommand{\bF}{{\mathbb F}}
\newcommand{\Fq}{\bF_q}
\newcommand{\lra}{\longrightarrow}
\newcommand{\fa}{\mathfrak a}
\newcommand{\fb}{\mathfrak b}
\newcommand{\fc}{\mathfrak c}

\newcommand{\cO}{\mathcal{O}}
\newcommand{\cF}{\mathcal{F}}

\newcommand{\cU}{\mathcal{U}}
\newcommand{\cV}{\mathcal{V}}

\newcommand{\cW}{\mathcal{W}}

\newcommand{\cP}{\mathcal{P}}

\title[The dynamical Mordell-Lang problem]{Periodic points, linearizing
  maps, and the dynamical Mordell-Lang problem}
\author{D.~Ghioca and T.~J.~Tucker}
\thanks{The second author was partially supported by NSA
    Grant 06G-067.}

\address{
Dragos Ghioca \\
Department of Mathematics \& Computer Science\\
University of Lethbridge \\
4401 University Drive \\ 
Lethbridge, Alberta T1K 3M4 
}

\email{dragos.ghioca@uleth.ca}

\address{
Thomas Tucker\\
Department of Mathematics\\
Hylan Building\\
University of Rochester\\
Rochester, NY 14627
}

\email{ttucker@math.rochester.edu}

\begin{document}

\begin{abstract}
  We prove a dynamical version of the Mordell-Lang conjecture for
  subvarieties of quasiprojective varieties $X$, endowed with the
  action of a morphism $\Phi:X\lra X$.  We use an analytic method based
  on the technique of Skolem, Mahler, and Lech, along with results of
  Herman and Yoccoz from nonarchimedean dynamics.  

\end{abstract}

\maketitle

\section{Introduction}
\label{intro}

Let $X$ be a quasiprojective variety over the complex numbers $\bC$, let $\Phi: X \lra
X$ be a morphism, and let $V$ be a closed subvariety of
$X$.  For any integer $i\geq 0$, denote by $\Phi^i$ the
$i^{\text{th}}$ iterate $\Phi\circ\cdots\circ\Phi$; for any point $\alpha\in X(\bC)$, we let $\cO_{\Phi}(\alpha):=\{\Phi^i(\alpha)\text{ : }i\in\N\}$ be the $\Phi$-orbit of $\alpha$.
If $\alpha \in X(\bC)$ has the property that there is some
integer $\ell\geq 0$ such that $\Phi^\ell(\alpha) \in W(\bC)$,
where $W$ is a periodic subvariety of $V$, then there are infinitely
many integers
$n\geq 0$ such that $\Phi^n(\alpha) \in V$.  More precisely, if
$M\geq 1$ is the period of $W$ (the smallest positive integer $j$ for which
$\Phi^j(W) = W$), then $\Phi^{kM + \ell}(\alpha) \in W(\bC)
\subseteq V(\bC)$ for integers $k\geq 0$.  It is natural
then to pose the following question.

\begin{question}\label{general}
  If there are infinitely many integers $m\geq 0$ such that
  $\Phi^m(\alpha) \in V(\bC)$, are there necessarily integers $M\geq
  1$ and $\ell\geq 0$ such that $\Phi^{kM + \ell}(\alpha) \in V(\bC)$
  for all integers $k\geq 0$?
\end{question}
Note that if $V(\bC)$ contains an infinite set of the form $\{\Phi^{kM+\ell}(\alpha)\}_{k\in\N}$ for some positive integers $M$ and $\ell$, then $V$ contains a positive dimensional subvariety invariant under $\Phi^M$ (simply take the union of the positive dimensional components of the Zariski closure of $\{\Phi^{kM+\ell}(\alpha)\}_{k\in\N}$).

Denis \cite{Denis-dynamical} appears to have been the first to pose Question~\ref{general}. He showed that the answer is ``yes'' under the
additional hypothesis that the integers $n$ for which $\Phi^n(\alpha)
\in V(\bC)$ are sufficiently dense in the set of all positive
integers; he also obtained results for automorphisms of projective
space without using this additional hypothesis.  Bell \cite{Bell}
later solved the problem completely in the case of automorphisms of
affine space, by showing that the set of all $n\in\N$ such that $\Phi^n(\alpha)\in V(\bC)$ is at most a finite union of arithmetic progressions.  More recently, results were obtained in the case when
$\Phi: \bA^2 \lra \bA^2$ takes the form $(f,g)$ for $f,g \in \bC[t]$
(\cite{Mike}) and the subvariety $V$ is a line, and in the case when
$\Phi:\bA^g \lra \bA^g$ has the form $(f,\dots,f)$ where $f \in K[t]$ (for a number field $K$)
has no periodic critical points other than the point at infinity
(\cite{Par}).

The technique used in \cite{Par} and \cite{Bell} is a modification of
a method first used by Skolem \cite{Skolem} (and later extended by
Mahler \cite{Mahler-2} and Lech \cite{Lech}) to treat linear
recurrence sequences.  The idea is
to show that there is a positive integer $M$ such that for each $i = 0,\dots, M-1$ there is
an integer $j \equiv i \pmod{M}$ and a $p$-adic analytic map
$\theta_j$ on the closed unit disc in $\bC_p$ such that $\theta_j(k) =
\Phi^{kM + j}(\alpha)$ for all $k \in \bN$.  Given any polynomial
$F$ in the vanishing ideal of $V$, one thus obtains a $p$-adic
analytic function $F \circ \theta_j$ that vanishes on all $k$ for which
$\Phi^{kM + j}(\alpha) \in V$.  Since an analytic function cannot have
infinitely many zeros in its domain of convergence unless that
function is identically zero, this implies that if there are
infinitely many $n \equiv i \pmod{M}$ such that $\Phi^n(\alpha) \in V$,
then $\Phi^{kM + j}(\alpha) \in V$ for {\it all} $k \in \N$.

In the case of \cite{Par}, the existence of the $p$-adic analytic maps
$\theta_j$ is proved by using linearizing maps developed by
Rivera-Letelier \cite{Riv}. One is able to show that the desired
analytic map exists at a prime $p$ provided that for any $\alpha_1, \dots,
\alpha_g \in K$, there exists a nonnegative integer $j$ such that $f^j(\alpha_i)$ is in a
$p$-adically indifferent periodic residue class modulo $p$, for each $i=1,\dots,g$.  It seems plausible, and even likely, that this
technique generalizes to the case of any map of the form $\Phi =
(f_1, \dots, f_g)$ where $f_i \in \bC[t]$.  Thus, we make the
following conjecture.  

\begin{conjecture}
\label{dynamical M-L}
Let $f_1,\dots,f_g\in\bC[t]$ be polynomials, let $\Phi$ be their
action coordinatewise on $\bA^g$, let $\cO_\Phi((x_1,\dots,x_g))$
denote the $\Phi$-orbit of
the point $(x_1,\dots,x_g)\in \bA^g(\bC)$, and let $V$ be a subvariety
of $\bA^g$.  Then $V$ intersects $\cO_\Phi((x_1,\dots,x_g))$ in at
most a finite union of orbits of the form
$\cO_{\Phi^M}(\Phi^{\ell}(x_1,\dots,x_g))$, for some nonnegative
integers $M$ and $\ell$.
\end{conjecture}
Note that the orbits for which $M=0$ are singletons, so that the
conjecture allows not only infinite forward orbits but also finitely
many extra points. We view our Conjecture~\ref{dynamical M-L} as a
dynamical version of the classical Mordell-Lang conjecture, where
subgroups of rank one are replaced by orbits under a morphism. Note
that when each $f_i$ is a monomial, Conjecture~\ref{dynamical M-L}
reduces to a dynamical formulation of the classical Mordell-Lang
conjecture for endomorphisms of the multiplicative group (see also our
Theorem~\ref{endomorphism}).

In this paper, we describe a more general framework for approaching
Conjecture~\ref{dynamical M-L}.  Note that when $\Phi$ takes the form $(f_1,
\dots, f_g)$ for $f_i \in \bC[t]$, the Jacobian of $\Phi$ is always
diagonalizable.  Here we prove more general results about neighborhoods
of fixed points with diagonalizable Jacobians.  Our principal tool is
work of Herman and Yoccoz \cite{HY} on linearizing maps for general
diffeomorphism in higher dimensions. Our first result is the following.

\begin{theorem}
\label{general attracting}
Let $p$ be a prime number, let $X$ be a quasiprojective variety defined over $\Cp$, and let $\Phi : X \lra X$ be a morphism defined over $\Cp$.
Let $\alpha \in X(\bC_p)$, and let $V$ be a closed subvariety of $X$ defined over $\Cp$. Assume
the $p$-adic closure of the orbit $\cO_{\Phi}(\alpha)$ contains a
$\Phi$-periodic point $\beta$ of period dividing $M$ such that $\beta$ and all of
its iterates are nonsingular, and such that the Jacobian of $\Phi^M$ at
$\beta$ is a nonzero homothety of $p$-adic absolute value less than one.  Then $V(\Cp)\cap\cO_{\Phi}(\alpha)$ is at most a finite union of orbits of the
form $\cO_{\Phi^k}(\Phi^{\ell}(\alpha))$, for some nonnegative
integers $k$ and $\ell$.
\end{theorem}

As a special case of Theorem~\ref{general attracting}, we derive
the main result of \cite{p-adic}, which we state here as
Theorem~\ref{polydyn}; this may be thought of as a special
case of Conjecture~\ref{dynamical M-L}.  Before stating
Theorem~\ref{polydyn}, we recall the definition of attracting periodic
points for rational dynamics.

\begin{definition}
\label{periodicity}
If $K$ is a field, and $\varphi\in K(t)$ is a
rational function,
then $z\in\bP^1(\Kbar)$ is a {\em periodic point} for
$\varphi$ if there exists an integer $n\geq 1$ such that
$\varphi^n(z) = z$.
The smallest such integer $n$ is the {\em period}
of $z$, and $\lambda=(\varphi^n)'(z)$ is
the {\em multiplier} of $z$.
If $|\cdot |_v$ is an absolute value on $K$,
and if $0<|\lambda|_v < 1$, then $z$ is called {\em attracting}.
\end{definition}

\begin{theorem}
\label{polydyn}
Let $g\ge 1$, let $p$ be a prime number, let
$\phi_1,\dots,\phi_g\in\Cp(t)$ be rational functions, and let
$\Phi:=(\phi_1,\dots,\phi_g)$ act coordinatewise on
$\left(\bP^1\right)^g$. Let $\alpha:=(x_1,\dots,x_g)\in
\left(\bP^1\right)^g(\Cp)$, and let $V\subset \left(\bP^1\right)^g$ be
a subvariety defined over $\Cp$. Assume the $p$-adic closure of the
orbit $\cO_{\Phi}(\alpha)$ contains an attracting $\Phi$-periodic
point $\beta:=(y_1,\dots,y_g)$ such that for some positive integer
$M$, we have $\Phi^M(\beta)=\beta$ and $(\phi_1^M)'(y_1)=\cdots =
(\phi_g^M)'(y_g)$.  Then $V(\Cp)\cap\cO_{\Phi}(\alpha)$ is at most a
finite union of orbits of the form
$\cO_{\Phi^k}(\Phi^{\ell}(\alpha))$, for some nonnegative integers $k$
and $\ell$.
\end{theorem}
The following result generalizes the approach of \cite{Bell}
and \cite{Par}. Essentially, it shows that if an orbit gets close to
some ``indifferent'' periodic point of $X$, then we can answer
Question~\ref{general} in the affirmative (see \cite{Riv} for the
terminology of ``indifferent'' points in the context of rational
dynamics).
\begin{theorem}
\label{general indifferent}
Let $X$ be a quasiprojective variety defined over a number field $K$,
let $\Phi:X\lra X$ be a morphism defined over $K$, and let $V$ be a
closed subvariety of $X$ defined over $K$. Let $\beta\in X(K)$ be a
periodic point of period dividing $M$ such that $\beta$ and its
iterates are all nonsingular points, and the Jacobian of $\Phi^M$ at
$\beta$ is a diagonalizable matrix whose eigenvalues
$\lambda_1,\dots,\lambda_g$ satisfy
\begin{equation}
\label{mult indep}
\prod_{j=1}^g \lambda_j^{e_j} \ne \lambda_i, 
\end{equation}
for each $1\le i\le g$, and any nonnegative integers $e_1,\dots,e_g$
such that $\sum_{j=1}^g e_j \ge 2$.

Then for all but finitely many primes $p$, there is a $p$-adic
neighborhood $\cV_p$ of $\beta$ (depending only on $p$ and $\beta$) such that if $\cO_{\Phi}(\alpha) \cap
\cV_p$ is nonempty for $\alpha \in X(\Cp)$, then
$V(\Cp)\cap\cO_{\Phi}(\alpha)$ is at most a finite union of orbits of
the form $\cO_{\Phi^k}(\Phi^{\ell}(\alpha))$, for some nonnegative
integers $k$ and $\ell$.
\end{theorem}

There are a variety of other issues to be dealt with but this seems,
at least, to represent evidence that Question \ref{general} has a
positive answer in general. Typically, the Jacobian at a point should
have distinct, nonzero, multiplicatively independent eigenvalues.
Thus, Theorem~\ref{general indifferent} says that if some iterate of
$\alpha$ belongs to a certain $p$-adic neighborhood $\cV_p$ of a
typical periodic point, then Question~\ref{general} has a positive
answer for $\Phi$ and $\alpha$.  While many obstacles remain -- most
notably the issue of the size of the neighborhood $\cV_p$ -- we are
hopeful that this approach will lead to a general answer for
Question~\ref{general}.  We note that Fakhruddin (\cite{Fa}) has shown
that periodic points are Zariski dense for a wide class of morphisms
of varieties; moreover, his methods show that in many cases over
number fields, there is a periodic point within any periodic residue
class at a finite prime.  Thus, it is reasonable to expect that some
iterate of $\alpha$ will be $p$-adically close to some periodic point. In conclusion, we propose the following general conjecture.
\begin{conjecture}
\label{general dynamical M-L}
Let $X$ be a quasiprojective variety defined over $\C$, let $\Phi:X\lra X$ be any morphism, and let $\alpha\in X(\C)$. Then for each subvariety $V\subset X$, the intersection $V(\C)\cap\cO_{\Phi}(\alpha)$ is a union of at most finitely many orbits of the form $\cO_{\Phi^k}(\Phi^{\ell}(\alpha))$, for some nonnegative integers $k$ and $\ell$.
\end{conjecture}

Finally, we prove Conjecture~\ref{general dynamical M-L} for algebraic group endomorphisms of
semiabelian varieties; in this case, one can show that the
argument from the proof of Theorem~\ref{general indifferent} always
applies.  In \cite{GT-IMRN}, this is proved using deep results of
Vojta \cite{V2} and Faltings \cite{Faltings} on integral points on
semiabelian varieties. Here, we give a purely dynamical proof.  

\begin{theorem}
\label{endomorphism}
Let $A$ be a semiabelian variety defined over a finitely generated subfield $K$ of $\bC$, and let $\Phi:A \lra A$ be an endomorphism defined over $K$. Then for every subvariety $V\subset A$ defined over $K$, and for every point $\alpha\in A(K)$, the intersection $V(K)\cap\cO_{\Phi}(\alpha)$ is at most a finite union of orbits of the form $\cO_{\Phi^k}(\Phi^{\ell}(\alpha))$ for some $k,\ell\in\N$.
\end{theorem}

Note that when $X$ is a semiabelian variety and $\Phi$
is a multiplication-by-$m$ map, this can likely be derived from
Faltings' proof \cite{Faltings} of the classical Mordell-Lang conjecture.  Such a
derivation seems less obvious, however, for more general
endomorphisms.

The idea of the proof of Theorems~\ref{general attracting} and
\ref{general indifferent} is fairly simple.  In both cases, we begin
by choosing an iterate $\Phi^\ell(\alpha)$ that is very close to $\beta$.  Work of Herman and Yoccoz \cite{HY} give a $p$-adic
function $h$ in a neighborhood of $\beta$ such that, for a suitable positive integer $M$, we have
$$
\Phi^M \circ h = h \circ A,$$
for some linear function $A$.  When
$A$ is a homothety, this means that iterates of $\Phi^\ell (\alpha)$
under $\Phi^M$ lie on an analytic line in $\bC_p^g$.  Composing with a polynomial in the vanishing ideal of the
subvariety $V$ gives a $p$-adic analytic function in one variable;
such a function is identically zero if it has infinitely many zeros.
Thus, for each congruence class $i = 0, \dots, M-1$ modulo $M$, either there are finitely may $n \equiv i \pmod{M}$ such that
$\Phi^n(\alpha) \in V$ or we have $\Phi^{n}(\alpha) \in V$ for all $n\ge \ell$ such that
$n\equiv i \pmod{M}$.  Under the conditions of Theorem~\ref{general
  indifferent}, it is necessary to take $p$-adic logarithms of iterates
in order to get a line in $\bC^g$ but otherwise the proof is the same.
Note that existence of the map $h$ in Theorem~\ref{general
  indifferent} depends on Yu's \cite{Yu} results on linear forms in
$p$-adic logarithms, which only apply over number fields. Under the
conditions of Theorem~\ref{general attracting}, the map $h$ exists even
when the eigenvalues of $A$ are transcendental.

We note that the Drinfeld module analog of Conjecture~\ref{dynamical
  M-L} has been proved using techniques similar to the ones employed
in this paper (see \cite{GT-compositio}).  This provided a positive
answer to a conjecture proposed by Denis \cite{Denis-conjectures} in
the case of $\Phi$-submodules $\Gamma$ of rank $1$, where
$\Phi:=(\phi_1,\dots,\phi_g)$, each $\phi_i:\Fq[t]\lra
\End_K(\bG_a)$ is a Drinfeld module, and $K$ is a function field of
transcendence degree $1$ over $\Fq$. However,
Conjecture~\ref{dynamical M-L} seems much more difficult, since the
proof \cite{GT-compositio} over Drinfeld modules makes use of the fact
that the polynomials $\phi_i(a)$ are additive (for any $a\in\Fq[t]$);
in particular, this allows us to find points of $\Gamma$ which are
arbitrarily close to $0$ with respect to any valuation $v$ of $K$ at
which the points of $\Gamma$ are integral and $\Phi$ has good
reduction. If we let $0\ne P\in\Fq[t]$ such that $|P|_v<1$, then $0$
is an attracting fixed point for the dynamical system $\Phi(P)$, and
each $\phi_i(P)$ has the same multiplier at $0$; thus the main result
of \cite{GT-compositio} is a special case of our
Theorem~\ref{polydyn}.

We now briefly sketch the plan of our paper. In
Section~\ref{attracting orbits} we prove Theorem~\ref{general
  attracting} and state a corollary of Theorem~\ref{polydyn}, while in
Section~\ref{indifferent orbits} we prove Theorems~\ref{general
  indifferent} and \ref{endomorphism}.

\begin{notation} 
  We write $\N$ for the set of nonnegative integers. If $K$ is a
  field, we write $\overline{K}$ for an algebraic closure of $K$.
  Given a prime number $p$, the field $\Cp$ will denote the completion
  of an algebraic closure $\overline{\bQ}_p$ of $\bQ_p$, the field of
  $p$-adic rationals.  We denote by $\mid\cdot\mid_p$ the usual
  absolute value on $\Cp$; that is, we have $|p|_p = 1/p$.  When we
  work in $\Cp^g$ with a fixed coordinate system, then, for
  $\vec{\alpha} = (\alpha_1\dots, \alpha_g) \in\Cp^g$ and $r>0$, we
  write $\D(\vec{\alpha},r)$ for the open disk of radius $r$ in $\Cp^g$
  centered at $\alpha$.  More precisely, we have
  $$
  \D(\vec{\alpha},r) := \{ (\beta_1, \dots, \beta_g) \in \Cp^g \;
  \mid \; \max_i |\alpha_i - \beta_i|_p < r \}.$$
  Similarly, we let
  $\Dbar(\vec{\alpha},r)$ be the {\it closed} disk of radius $r$
  centered at $\vec{\alpha}$.  In the case where $g=1$, we drop the
  vector notation and denote our discs as $\D(\alpha,r)$ and
  $\Dbar(\alpha,r)$.  We say that a function $F$ is {\it rigid analytic} on
  $\D(\alpha,r)$ (resp. $\Dbar(\alpha, r)$) if there is a power series
  $\sum_{n=0}^\infty a_n (z-\alpha)^n$, with coefficients in $\Cp$,
  convergent on all of $\D(\alpha,r)$ (resp.  $\Dbar(\alpha, r)$) such
  that $F(z) = \sum_{n=0}^\infty a_n (z-\alpha)^n$ for all $z \in
  \D(\alpha,r)$ (resp. $\Dbar(\alpha, r)$).

\end{notation}

\section{Attracting points}
\label{attracting orbits}

We will begin with a theorem of Herman and Yoccoz \cite{HY} on
linearization of analytic maps near one of their fixed points. First
we set up the notation.  Let $\vec{0}$ be the zero vector in $\Cp^g$, and for $\vec{x}:=(x_1,\dots,x_g)$ we let
\begin{equation}\label{f}
 f(\vec{x}) =  \sum_{(i_1,\dots,i_g) \in \bN^g}
b_{i_1,\dots,i_g} x_1^{i_1} \cdots x_g^{i_g}
\end{equation}
be a power series over $\Cp$ which fixes $\vec{0}$, and it has a positive radius of convergence; i.e., there is some $r > 0$ such that \eqref{f} converges on $\D(\vec{0},r)$.
Furthermore, we assume there exists $A\in\GL(g,\Cp)$ such that
$$f(\vec{x}) = A\cdot \vec{x} + \; \text{higher order terms }.$$
In
this case, $f$ is a formal diffeomorphism in the terminology of
\cite{HY}.  More generally, for a formal power series $\psi$ in
$\bC_p^g$ centered at $\vec{\alpha}$, we define $D\psi_{\vec{\alpha}}$
to be the linear part of the power series.  Thus $Df_{\vec{0}} = A$.
Note that this coincides with the usual definition of the $D$-operator
from the theory of manifolds (that is, $D \psi_{\vec{\alpha}}$ is the usual
Jacobian of $\psi$ at $\vec{\alpha}$ -- see \cite[I.1.5]{Jost}).

Let $\lambda_1, \dots, \lambda_g$ be the eigenvalues of $A$.  Suppose
that there are constants $C, b > 0$ such that
\begin{equation}\label{C}
| \lambda_1^{e_1} \cdots \lambda_g^{e_g} - \lambda_i|_p \geq C
\left( \sum_{j=1}^g e_j \right)^{- b}
\end{equation}
for any $1 \leq i \leq g$ and any tuple $(e_1, \dots, e_g) \in \bN^g$ such
that $\sum_{j=1}^g e_j \geq 2$ (this is condition (C) from page
413 of \cite{HY}). Note that \eqref{C} already implies that no $\lambda_i=0$.

The following result is Theorem $1$ of \cite{HY}.  
\begin{theorem}\label{key}
  Let $f$ and $A$ be as above.  There exists $r>0$, and there exists a bijective, $p$-adic analytic function $h:\D(\vec{0},r)\lra \D(\vec{0},r)$ such that
\begin{equation}
\label{linearization}
f ( h(\vec{x})) = h ( A\vec{x}),
\end{equation}
for all $\vec{x}\in \D(\vec{0},r)$,
where $Dh_{\vec{0}} = \Id$.  
\end{theorem}

Before continuing, we need to define $D \Phi$ more generally for
$\Phi$ a morphism of a quasiprojective variety.  If $X$ is a
quasiprojective variety defined over a field $L$, and $\Phi:X\lra X$ is
a morphism, and both $\beta$ and $\Phi(\beta)$ are nonsingular points
in $X(L)$, then $\Phi$ induces an $L$-linear map
$$D\Phi_{\beta}: T_\beta \lra T_{\Phi(\beta)}$$
where $T_\beta$ is the
stalk of the tangent sheaf for $X$ at $\beta$. Since $\beta$ and
$T(\beta)$ are nonsingular, both $T_\beta$ and $T_{\Phi(\beta)}$ are
vector spaces of dimension $\dim X$ over $L$ (see \cite[II.8]{H}).
Note that when $L = \bC$ and $\beta$ and $\Phi(\beta)$ are in a
coordinate patch $\cU$ on the complex manifold $X^{\nonsing}(\bC)$,
then $D \Phi_{\beta}$ can be written in coordinates using the partial
derivatives of $\Phi$ with respect to these coordinates (i.e.  as the
Jacobian matrix of $\Phi$ expressed with respect to these
coordinates).

We are ready to prove Theorem \ref{general attracting}.  

\begin{proof}
  Let $\lambda\in\Cp$ such that $D\Phi^M_{\beta}=\lambda\cdot\Id$;
  according to our hypotheses, we have $0<|\lambda|_p < 1$. Let
  $j\in\{0,\dots,M-1\}$ be fixed.  Using the fact that $X(\bC_p)$ is a
  $p$-adic analytic manifold in a neighborhood of each iterate
  $\Phi^j(\beta)$ we may find an analytic function $\cF_j$ defined on
  a sufficiently small neighborhood $\cU_j$ of $\vec{0} \in \Cp^g$
  which maps $\cU_j$ bijectively onto a neighborhood $\cV_j$ of
  $\Phi^j(\beta)$. Then we write $\Psi_j:=\cF_j^{-1}\circ\Phi^M\circ
  \cF_j$ as a function of the following form (note that
  $D(\Psi_j)_{\vec{0}} = \lambda \cdot \Id$):
  $$\Psi_j(\vec{x}):=(\cF_j^{-1}\circ\Phi^M\circ \cF_j)(\vec{x}) = \lambda \cdot \vec{x} + \; \text{higher order terms }.$$
Since $|\lambda^i - \lambda|_p =
|\lambda|_p$ for $i \geq 2$, we see that \eqref{C} is satisfied; so we
have a bijective analytic function $h_j: \D(\vec{0},r_j)\lra \D(\vec{0},r_j)$ (for some $r_j>0$) such that
$$\Psi_j \circ h_j = h_j \circ \lambda \Id$$
and $D\left(h_j\right)_{\vec{0}} = \Id$, by Theorem \ref{key}. Let $r>0$ such that for each $0\le j\le M-1$, we have
\begin{equation}
\label{one neighborhood in all}
(\Phi^j\circ \cF_0)(\D(\vec{0},r)) \subset \cF_j(\D(\vec{0},r_j)).
\end{equation}

Let $N_0$ be the smallest positive integer such that $\Phi^{N_0}(\alpha)\in \cF_0(\D(\vec{0}, r))$; then $\Phi^{N_j}(\alpha)\in \cF_j(\D(\vec{0},r_j))$, where $N_j:=N_0+j$ for each $j=1,\dots,M-1$.
Let $\vec{\alpha_j}\in \D(\vec{0},r_j)$ satisfy $h_j(\vec{\alpha_j})=\cF_j^{-1}(\Phi^{N_j}(\alpha))$.
Note
that since $|\lambda|_p < 1$, we have
\begin{equation}\label{k}
(\cF_j^{-1}\circ \Phi^{kM})\left(\Phi^{N_j}(\alpha)\right) = (\Psi_j^k\circ h_j)(\vec{\alpha_j}) = h_j(\lambda^k \cdot \vec{\alpha_j}).
\end{equation}

Now, for each polynomial $F$ in the vanishing ideal of $V$, we
construct the function $\Theta_j:\Dbar(0,1)\lra \Cp$ given by
$$\Theta_j(z):= F\left( (\cF_j \circ h_j)\left(z\cdot \vec{\alpha_j}\right)\right).$$
The function $\Theta_j$ is analytic because each $h_j$ is analytic on $\D(\vec{0},r_j)$, and $\vec{\alpha_j}\in \D(\vec{0},r_j)$.

For each $k\in\N$ such that $\Phi^{kM+N_j}(\alpha)\in V(\Cp)$, we have $\Theta_j(\lambda^k)=0$. 
Because $\lim_{k\to\infty}\lambda^k =0$, we conclude that if there are infinitely many $k$ such that
$\Phi^{N_0+j+Mk}(\alpha) \in V(\bC_p)$ then $\Theta_j$ is identically equal to $0$; hence $\Theta_j(\lambda^k) = 0$ for \emph{all} $k\in\N$, which means that
$\Phi^{N_0+j+Mk}(\alpha) \in V(\bC_p)$ for \emph{all} $k$. Thus we have
\begin{eqnarray}
\label{dichotomy}
\text{either } \Phi^{N_0+ j + Mk}(\alpha) \in V(\Cp) \text{ for all $k \in \N$}\\
\nonumber
\text{or } V(\Cp)\cap\cO_{\Phi^M}\left(\Phi^{N_0+j}(\alpha)\right)\text{ is finite.}
\end{eqnarray}

Since
$$\cO_{\Phi}(\alpha)=\{\Phi^i(\alpha)\text{ : } 0\le i\le N_0-1\}\bigcup\left(\bigcup_{j=0}^{M-1}\cO_{\Phi^M}(\Phi^{N_0+j}(\alpha))\right),$$
we conclude the proof of Theorem~\ref{general attracting}.
\end{proof}

The following result is an immediate corollary of Theorem~\ref{polydyn}.
\begin{cor}
\label{same map attracting}
Let $f\in\Cp[t]$ be defined by $f(t) = \sum_{i=1}^m a_i t^i$, where
for each $i=1,\dots,m$, we have $|a_i|_p\le 1$, while $0<|a_1|_p<1$.
We consider $\Phi$ the coordinatewise action of $f$ on $\bA^g$ (where
$g\ge 1$). Let $\alpha=(x_1,\dots,x_g)\in\bA^g(\Cp)$ satisfy
$|x_i|_p<1$ for each $i=1,\dots,g$.  Then for any subvariety
$V\subset\bA^g$ defined over $\Cp$, the intersection
$V(\Cp)\cap\cO_{\Phi}(\alpha)$ is either finite, or it contains
$\cO_{\Phi}(\Phi^{\ell}(\alpha))$ for some $\ell\in\N$.
\end{cor}

Furthermore, both Theorem~\ref{general attracting} and
Theorem~\ref{polydyn} (with its Corollary~\ref{same map attracting})
can be made effective, as one can use Newton polygons to find the
zeros of $p$-adic analytic functions.  

\section{Indifferent points}
\label{indifferent orbits}

Using the same set-up as in the beginning of Section~\ref{attracting
  orbits}, we prove Theorem~\ref{general indifferent}. But we first
need to define an arbitrary power $J^z$ of a Jordan matrix, whenever
$z\in\D(0,p^{-1/(p-1)})$.
\begin{proposition}
\label{Jordan}
Let $J\in\GL(g,\Cp)$ be a Jordan matrix with the property that each
eigenvalue $\lambda_i$ of $J$ satisfies $|\lambda_i|_p=1$. Then there
exists a positive integer $d$ such that we may define $(J^d)^z$ for
each $z\in\D(0,p^{-1/(p-1)})$ satisfying the following properties:
\begin{enumerate}
\item for each $k\in\N$, the matrix $(J^d)^k$ is the usual $(dk)$-th power of $J$; and
\item the entries of $(J^d)^z$ are rigid analytic functions of $z$.
\end{enumerate}
\end{proposition}

\begin{proof}
We let $d$ be a positive integer such that $|\lambda_i^d-1|_p<1$ for each $i=1,\dots,g$. 
We define the power $(J^d)^z$ for each $z\in\D(0,p^{-1/(p-1)})$ on each Jordan block.
For each diagonal Jordan block, we define $(J^d)^z$ as the diagonal matrix whose eigenvalues are the corresponding $(\lambda_i^d)^z$ for each $z\in\D(0,p^{-1/(p-1)})$.

Assume now that $J_0$ is a Jordan block of dimension $m\le g$ corresponding to an eigenvalue $\lambda$. Then for each $z\in\D(0,p^{-1/(p-1)})$, we define $(J_0^d)^z$ be the following matrix
\[(\lambda^d)^z\cdot \left(\begin{array}{ccccc}
1 & \frac{z}{1!\cdot\lambda} & \frac{z(z-1)}{2! \cdot\lambda^2}  & \dots & \frac{z(z-1)\cdots (z-m+1)}{m!\cdot \lambda^{m-1}}\\
0 & 1 &  \frac{z}{1!\cdot\lambda} & \dots & \frac{z(z-1)\cdots (z-m+2)}{(m-1)!\cdot \lambda^{m-2}}\\
0 & 0 & 1 & \dots & \frac{z(z-1)\cdots (z-m+3)}{(m-2)!\cdot \lambda^{m-3}}\\
\dots & \dots & \dots & \dots & \dots \\
0 & 0 & 0 & \dots & 1
\end{array}\right).
\]
It is immediate to check that $(J_0^d)^z$ equals the usual power
$J_0^{dk}$ whenever $k\in\N$. Furthermore, since $|\lambda^d-1|_p<1$,
we see that $(\lambda^d)^z = \exp_p ( (\log_p(\lambda^d))z)$ (where
$\exp_p$ and $\log_p$ are the usual $p$-adic exponential and logarithm
functions) is a rigid analytic function of $z\in\D(0,p^{-1/(p-1)})$
(see \cite[Section 5.4.1]{Robert}).  Hence each entry of $(J_0^d)^z$
is a rigid analytic function of $z\in\D(0,p^{-1/(p-1)})$ of as well.
\end{proof}

\begin{proof}[Proof of Theorem~\ref{general indifferent}.]
  Let $B\in\GL(g,\Kbar)$ such that $B (D(\Phi^M)_{\beta}) B^{-1}$ is a diagonal
  matrix $\Lambda$; at the expense of replacing $K$ by a finite extension, we may assume $B\in\GL(g,K)$. Then for all but finitely many primes $p$, the
  entries of both $B$ and $B^{-1}$ have $p$-adic absolute values at
  most equal to $1$. Let $\lambda_i$ (for $1\le i\le g$) be the
  eigenvalues of $D(\Phi^M)_{\beta}$. According to our hypotheses, each
  $\lambda_i$ is nonzero. Thus for all but finitely many primes $p$,
  each $\lambda_i$ is a $p$-adic unit.
  
  Fix a prime $p$ satisfying the above conditions, and we fix an
  embedding of $K$ into $\Cp$. Let $j\in\{0,\dots,M-1\}$ be fixed.
  Clearly, we have $D(\Phi^M)_{\Phi^j(\beta)}=D(\Phi^M)_{\beta}$.
  Since each $\Phi^j(\beta)$ is a nonsingular point, there exists a
  sufficiently small neighborhood $\cU_j\subset \Cp^g$ of $\vec{0}$,
  and an analytic function $\cF_j$ that maps $\cU_j$ bijectively onto
  a small neighborhood of $\Phi^j(\beta)\in X(\Cp)$; let
  $\Psi_j:=\cF_j^{-1}\circ \Phi^M \circ \cF_j$. Then
$$\Psi_j(\vec{x})=(B^{-1}\Lambda B)\cdot \vec{x} + \; \text{higher order terms }.$$
Using hypothesis \eqref{mult indep}
  and \cite[Theorem $1$]{Yu}, we conclude that \eqref{C} is satisfied
  by the eigenvalues $\lambda_i$. Using Theorem~\ref{key}, we
  conclude that there exists a positive number $r_j>0$ such that $\D(\vec{0},r_j)\subset\cU_j$, and there exists
  a bijective analytic function $h_j:\D(\vec{0},r_j)\lra \D(\vec{0},r_j)$ such
  that $\Psi_j\circ h_j = h_j\circ (B^{-1}\Lambda B)$.

Let $r$ be a positive number such that for every $j=0,\dots,M-1$ we have
$$\Phi^j(\cF_0(\D(\vec{0},r))) \subset \cF_j(\D(\vec{0},r_j)).$$
We
let $\cV_p:=\cF_0(\D(\vec{0},r))$ be the corresponding $p$-adic
neighborhood of $\beta$ in $X(\Cp)$. Suppose that $\cO_{\Phi}(\alpha)
\cap \cV_p$ is nonempty.  Then there exists $N_0\in\N$ such that
$\Phi^{N_0}(\alpha)\in \cV_p$, and so, $\Phi^{N_j}(\alpha)\in
\cF_j(\D(\vec{0},r_j))$ where $N_j:=N_0+j$, for each
$j=0,\dots,M-1$. Let $\vec{\alpha_j}\in \D(\vec{0},r_j)$ such that
$h_j(\vec{\alpha_j})=\cF_j^{-1}(\Phi^{N_j}(\alpha))$. Then for each
$k\in\N$, we have
$$\Phi^{kM+N_j}(\alpha)=(\cF_j\circ
h_j)\left(B^{-1}\Lambda^{k}B(\vec{\alpha_j})\right).$$
Note that
$B^{-1}\Lambda^{k}B(\vec{\alpha_j})$ is in $\D(\vec{0},r_j)$ for each
$k\in\N$, since each entry of $B$, $B^{-1}$, and $\Lambda$ is in
$\Dbar(0,1)$,

Let $d$ be a positive integer as in the conclusion of
Proposition~\ref{Jordan}. Then the entries of the matrix
$(\Lambda^d)^z$ are $p$-adic analytic functions of $z$ in the disk
$\D(0,p^{-1/(p-1)})$.  Hence, the entries of the matrix
$(\Lambda^d)^{2pz}$ are $p$-adic analytic functions of
$z\in\Dbar(0,1)$ (since $|2pz|_p < p^{-1/(p-1)}$ for $|z|_p \leq
1$).  Therefore, for each fixed $\ell=0,\dots,2pd-1$, the entries of the
matrix $\Lambda^{\ell}\cdot (\Lambda^d)^{2pz}$ are $p$-adic analytic
functions of $z\in\Dbar(0,1)$.
  
  Let $F$ be any polynomial in the vanishing ideal of $V$. Then, for
  each $j=0,\dots,M-1$ and for each $\ell=0,\dots,2pd-1$, the function
  $$\Theta_{j,\ell}:\Dbar(0,1)\lra \Cp$$
  defined by
$$\Theta_{j,\ell}(z) = F\left((\cF_j\circ h_j)\left( B^{-1}
    \left(\Lambda^{\ell}\cdot (\Lambda^d)^{2pz}\right)
    B(\alpha_j)\right)\right)$$
is rigid analytic. Furthermore, for
each $k\in\N$ such that 
$$\Phi^{N_0+j+M(2kpd+\ell)}(\alpha)\in V(\Cp)$$
we obtain that $\Theta_{j,\ell}(k)=0$. Because the zeros of a nonzero
rigid analytic function cannot accumulate (see \cite[Section
$6.2.1$]{Robert}), we conclude that
\begin{eqnarray}
\label{dichotomy 2}
\text{either } \Phi^{N_0+j+M(2kpd+\ell)} (\alpha) \in V(\Cp) \text{ for all $k \in \N$}\\
\nonumber
\text{or }
V(\Cp)\cap\cO_{\Phi^{2Mpd}}\left(\Phi^{N_0+j+M\ell}(\alpha)\right) \text{ is finite.}
\end{eqnarray}
Since \eqref{dichotomy 2} holds for each $j=0,\dots,M-1$ and for
each $\ell=0,\dots,2pd-1$, this concludes the proof of
Theorem~\ref{general indifferent}.
\end{proof}

We prove now Theorem~\ref{endomorphism}. 

\begin{proof}[Proof of Theorem~\ref{endomorphism}.]
  For each $n\ge 0$, we let $A_n:=\Phi^n(A)$ (where $A_0=A$). Clearly,
  $A_{n+1}\subset A_n$ for each $n\in\N$. Also, each $A_n$ is
  connected because it is the image of a connected group through a
  morphism; hence, each $A_n$ is a semiabelian variety itself. On the
  other hand, there is no infinite descending chain of semiabelian
  varieties; therefore there exists $N\in\N$ such that $A_n=A_N$ for
  each $n\ge N$, and so, $\Phi$ restricted to $A_N$ is an isogeny.
  Clearly, it suffices to prove Theorem~\ref{endomorphism} after
  replacing $\alpha$ by $\Phi^{N}(\alpha)$. Thus, at the expense of
  replacing $A$ by $A_N$, and $V$ by $V\cap A_N$, we may assume that
  $\Phi$ is an isogeny.
  
  We will proceed by induction on the dimension of $V$.  The case
  $\dim V = 0$ is trivial.  Using the inductive hypothesis, we may
  show that to prove our result for orbits of a point $\alpha$ it suffices
  to prove it for orbits of a multiple $m \alpha$ of $\alpha$.
  Indeed, fix a positive integer $m$ and suppose that Theorem~\ref{endomorphism} is
  true for orbits of $m \alpha$. Then, given any subvariety $V$, we
  know that the set of $n$ for which $\Phi^n(m \alpha) \in mV$ forms a
  finite union $\cP$ of arithmetic progressions. Thus, if we let $W$ equal
  the inverse image of $mV$ under the multiplication-by-$m$ map we
  know that the set of $n$ such that $\Phi^n(\alpha) \in W$ forms the
  same finite union $\cP$ of arithmetic progressions, since $\Phi$ is a
  group endomorphism. We let $Z_1,\dots,Z_s$ be the positive dimensional irreducible components of the Zariski closure of $\{\Phi^n(\alpha)\}_{n\in\cP}$ (if $\{\Phi^n(\alpha)\}_{n\in\cP}$ is finite, then also $V(K)\cap\cO_{\Phi}(\alpha)$ is finite, and so, we are done). Hence each $Z_i$ is a $\Phi$-periodic subvariety, and all but finitely many of the $\Phi^n(\alpha)$ for $n\in\cP$ are contained in one of the $Z_i(K)$.
Thus, we need only show that for each $i$, the set of $n$ such that $\Phi^n(\alpha) \in V
  \cap Z_i$ forms a finite union of arithmetic progressions.  Each
  $Z_i$ is contained in one of irreducible components of $W$ and thus
  $\dim Z_i \leq \dim V$.  If $Z_i$ is contained in $V$, then the set
  of $n$ for which $\Phi^n(\alpha) \in Z_i = Z_i \cap V$ is a
  finite union of arithmetic progressions, since $Z_i$ is
  $\Phi$-periodic.  If $Z_i$ is not contained in $V$, then $Z_i \cap
  V$ has dimension less than $\dim Z_i \leq \dim V$ and the set of $n$
  such that $\varphi^n(\alpha) \in Z_i \cap V$ is a finite union of
  arithmetic progressions by the inductive hypothesis.

  Let $\exp:\bC^g\lra A(\bC)$ be the usual exponential map associated
  to a semiabelian variety (where $g:=\dim A$). Then there exists a
  linear map $\varphi:\bC^g\lra \bC^g$ such that for each $z\in\bC^g$,
  we have
\begin{equation}
\label{uniformization endomorphism}
\Phi(\exp(z)) = \exp(\varphi(z)).
\end{equation}
At the expense of replacing $K$ by another finitely generated field,
we may assume $\varphi$ is defined over $K$. Let $L$ be the
corresponding nonsingular matrix such that $\varphi(z) = Lz$ (the
matrix is nonsingular because $\Phi$ is an isogeny).  We choose an
embedding over $\bC$ of $\iota:A \lra \bP^N$ as an open subset of a projective
variety (for some positive integer $N$). At the expense of enlarging $K$, we may assume the above embedding $\iota$ is defined over $K$. We write $\iota(A) = Z(\fa)\setminus Z(\fb)$ for homogeneous ideals
$\fa$ and $\fb$ in $K[x_0, \dots, x_N]$, where $Z(\fc)$ denotes the Zariski closed subset of $\bP^N$ on which the ideal $\fc$ vanishes.  We choose
generators $F_1, \dots, F_m$ and $G_1, \dots, G_n$ for $\fa$ and $\fb$
respectively.  We let $\add: A \times A \lra A$ denote the addition map
and $\minus: A \lra A$ the inversion map, written with respect to our
chosen coordinates on $\bP^N$.  

\begin{claim}
\label{arithmetic}
There exists a prime number $p$, and an embedding of $K$ into $\bQ_p$ such that 
\begin{enumerate}
\item[(i)] there exists a model $\cA$ of $A$ over $\bZ_p$.
\item[(ii)] $\alpha \in \cA(\bZ_p)$.
\item[(iii)] $L$ is conjugate over $\mathbb{Z}_p$ to its Jordan canonical form $\Lambda$, and
moreover each of its eigenvalues $\lambda_i$ is a $p$-adic unit.
\item[(iv)] the maps $\Phi$ and $\minus$ extend as endomorphisms of
  the $\bZ_p$-scheme $\cA$, while $\add$ extends to a morphism between
  $\cA\times\cA$ and $\cA$.
\end{enumerate}
\end{claim}

\begin{proof}[Proof of Claim~\ref{arithmetic}.]
Let $R$ be a finitely generated subring of $K$ such that 
\begin{enumerate}
\item[(1)] the coefficients of $F_1,\dots, F_m$, $G_1,\dots,G_n$, and
  of the polynomials defining $\Phi$, $\add$, and $\minus$ are all
  contained in $R$;
\item[(2)] $\alpha\in\bP^N(R)$;
\item[(3)] $L$ is conjugate over $R$ to its Jordan canonical form
  $\Lambda$ and all of the eigenvalues of $\Lambda$ are in $R$.
\end{enumerate}

Let $\mathcal{D}$ be the Zariski dense open subset of
$\bP^N_{\Spec{R}}$ defined by $G_i\ne 0$ for at least one
$i\in\{1,\dots,n\}$. Then let $\mathcal{B}$ be the closed subset of
$\mathcal{D}$ defined by the zeros of each $F_j$ (for
$j\in\{1,\dots,m\}$) and of the polynomials defining $\Phi$. Because
$\Phi$ is an endomorphism of $A$, we conclude that $\mathcal{B}$ does
not intersect the generic fiber of $\mathcal{D}\lra \Spec(R)$.
Therefore $\mathcal{B}$ is contained in a finite union of special
fibers of $\bP^N_{\Spec(R)}\lra\Spec(R)$. Let $E_1$ be the proper
closed subset of $\Spec R$ corresponding to these special fibers.
Then, for any open affine subset $\Spec R'$ of $\Spec R \setminus E_1$, there exists a model $\cA_{R'}$ of $A$ over $\Spec(R')$, and moreover, $\Phi$ extends to a $R'$-morphism from $\cA_{R'}$ to itself. By the
same reasoning, there is a proper closed subset $E_2$ such that for any
open affine $\Spec R'$ of $\Spec R \setminus E_2$, the maps $\add$ and $\minus$ extend to morphisms
$$\add:\cA_{R'}\times\cA_{R'}\lra\cA_{R'}\text{
  and }\minus:\cA_{R'}\lra\cA_{R'}.$$ Similarly, the Zariski closure
of the point $\alpha$ in $\bP^n(R)$ only meets $Z(\fb)$ over primes
contained in some closed subset $E_3$.  Taking an element $f\in R$ outside
the prime ideals contained in $E_1 \cup E_2 \cup E_3$ and localizing $R$ at
$f$, we obtain an open affine $\Spec R_f$ of $\Spec R$ such that
$\Phi$, $\add$, and $\minus$ extend to maps of $R_f$-schemes and such
that $\alpha$ extends to an $R_f$-point on $\cA_{R_f}$.

Since $R_f$ is finitely generated as a ring, we may write $R_f =
\bZ[u_1, \dots, u_e]$ for some nonzero elements $u_i$.
By \cite[Lemma $3.1$]{Bell} (see also \cite{Lech}), there is a prime
$p$ such that the field of fractions of $R$ embeds into $\bQ_p$ in
such a way that all of the $u_i$, all of the eigenvalues of $\Lambda$,
and all of the reciprocals of the eigenvalues of $\Lambda$ are sent to
elements of $\bZ_p$.  Note that such a map gives an embedding of $R_f$
into $\bZ_p$.  Extending the base of $\cA_{R_f}$, $\Phi$, $\add$, and
$\minus$ from $R_f$ to $\bZ_p$ via this embedding yields the model and the
maps with the desired properties.
\end{proof}

Let $p$ be a prime number for which the conclusion of Claim~\ref{arithmetic} holds.
Therefore, we have a model $\cA$ of $A$ over $\bZ_p$ such that $\alpha\in\cA(\bZ_p)$. Furthermore, both $Z(\fa)$ and $Z(\fb)$ have models $\cV$ and $\cW$ over $\bZ_p$, and so,
$$\cA(\bZ_p)=\cV(\bZ_p) \cap \left(\bP^N\setminus \cW\right)(\bZ_p)$$ 
is compact because it is the intersection of two compact subsets of $\bP^N(\bZ_p)$. Indeed, $\cV(\bZ_p)$ is compact because it is a closed subset of the compact set $\bP^N(\bZ_p)$. On the other hand, $\left(\bP^N\setminus \cW\right)(\bZ_p)$ consists of finitely many
residue classes of $\bP^N(\bZ_p)$ and thus, it is compact because $\bZ_p$ is compact. The above finitely many residue classes correspond to points in $\left(\bP^N\setminus Z(\bar{\fb})\right)(\mathbb{F}_p)$, where $\bar{\fb}$ is the ideal of $\mathbb{F}_p[x_0,\dots,x_N]$ generated by the reductions modulo $p$ of each $G_i$.

Let $d$ be a positive integer as in the conclusion of
Proposition~\ref{Jordan} corresponding to $\Lambda$. Let
$(L^d)^z=B^{-1}(\Lambda^d)^zB$, where $L=B^{-1}\Lambda B$. Using
Proposition~\ref{Jordan}, we conclude that $z\mapsto (L^d)^z$ is an
analytic function whenever $z\in\D(0,p^{-1/(p-1)})$ (with values in
$\GL(g,\Cp)$).  Therefore, for each fixed vector $\vec{x_0}\in\Cp^g$,
\begin{equation}
\label{is analytic}
\text{the function }z\mapsto (L^d)^z\cdot \vec{x_0}\text{ is rigid analytic whenever $z\in\D(0,p^{-1/(p-1)})$.}
\end{equation}

Let $\D(\vec{0},r)$ be a sufficiently small neighborhood of the origin
in $\Cp^g$ such that $\exp$ is bijective $p$-adic analytic on
$\D(\vec{0},r)$ (the power series for the exponential map over $\Cp$
has positive radius of convergence by \cite[III.7.2]{NB}).  After
replacing $\alpha$ by $m \cdot \alpha$, for a positive integer $m$ (as
we may, by the remarks at the beginning of the proof), we may assume
$\alpha\in\exp(\D(\vec{0},r))$.  To see this, we note that
$\D(\vec{0},r)$ is an additive subgroup because $|\cdot|_p$ is
nonarchimedean, so its image in $A(\bQ_p)$ is an open subgroup since
$\exp$ is bijective and analytic on $D(\vec{0},r)$.  The fact that
$\cA(\bZ_p)$ is compact means that any open subgroup of $\cA(\bZ_p)$ has
finitely many cosets in $\cA(\bZ_p)$.  Because $\alpha\in\cA(\bZ_p)$ there is a positive integer $m$ such that $m  \alpha
\in \exp(\D(\vec{0},r))$.

Let $\vec{\beta}\in\D(\vec{0},r)$ such that
$\exp(\vec{\beta})=\alpha$.  Since the coefficients of $L$ are all
$p$-adic integers, it follows that $L^n \vec{\beta}$ is in
$\D(\vec{0},r)$ for any $n\in\N$. Let
$\vec{\beta_j}:=L^j(\vec{\beta})$, for each $j=0,\dots,2pd-1$.

The remainder of the argument now proceeds as in the proof of
Theorem~\ref{general indifferent}.  Fix $j\in\{0,\dots,2pd-1\}$. For
each $k\in\N$, we have
$$\Phi^{j+2pkd}(\alpha)=\exp\left( (L^{d})^{2pk} \cdot
  \vec{\beta_j}\right),$$
For each polynomial $F$ in the vanishing
ideal of $V$, we define the function $\Theta_j$ on
$\Dbar(0,1)$ by
$$\Theta_j(z) = F\left(\exp\left( (L^{d})^{2pz} \cdot
    \vec{\beta_j}\right)\right).$$
Using \eqref{is analytic} (along with the fact that $|2p z|_p <
p^{-1/(p-1)}$ for any $|z|_p \leq 1$), we
conclude that $\Theta_j$ is rigid analytic, and so, assuming that
there are infinitely many $k\in\N$ such that $\Phi^{j+2pkd}(\alpha)\in
V(\Cp)$, we see that $\Theta_j$ is identically equal to $0$. We
conclude that
\begin{eqnarray}
\label{dichotomy 3}
\text{either } \Phi^{ j + 2pkd}(\alpha) \in V(\Cp) \text{ for all $k \in \N$}\\
\nonumber
\text{or } V(\Cp)\cap\cO_{\Phi^{2pd}}(\Phi^{j}(\alpha))\text{ is finite.}
\end{eqnarray}
This concludes the proof of Theorem~\ref{endomorphism}.
\end{proof}

As mentioned in Section~\ref{attracting orbits}, it should be possible
to make effective Theorems~\ref{general indifferent} and
\ref{endomorphism}.  In particular, this should allow for the kinds of
explicit computations that Flynn-Wetherell \cite{FW}, Bruin
\cite{Bruin}, and others performed in the context of rational points
on curves of genus greater than one.  It may also be possible to
extend the proof of Theorem~\ref{endomorphism} to the case of any
finite self-map of a nonsingular variety with trivial canonical
bundle.


\def\cprime{$'$} \def\cprime{$'$} \def\cprime{$'$} \def\cprime{$'$}
\providecommand{\bysame}{\leavevmode\hbox to3em{\hrulefill}\thinspace}
\providecommand{\MR}{\relax\ifhmode\unskip\space\fi MR }
\providecommand{\MRhref}[2]{%
  \href{http://www.ams.org/mathscinet-getitem?mr=#1}{#2}
}
\providecommand{\href}[2]{#2}

\end{document}